\documentclass[11pt]{article}
\usepackage[numbers,sort&compress]{natbib}
\usepackage{enumerate}
\usepackage{amscd}
\usepackage{amsmath}
\usepackage{array}
\usepackage{latexsym}
\usepackage{amsfonts}
\usepackage{amssymb}
\usepackage{amsthm}
\usepackage{verbatim}
\usepackage{mathrsfs}
\usepackage{enumerate}
\usepackage{hyperref}

\theoremstyle{plain}\newtheorem{definition}{Definition}[section]
\theoremstyle{definition}\newtheorem{theorem}{Theorem}[section]
\theoremstyle{plain}\newtheorem{lemma}[theorem]{Lemma}
\theoremstyle{plain}\newtheorem{coro}[theorem]{Corollary}
\theoremstyle{plain}
\theoremstyle{remark}\newtheorem{remark}{Remark}[section]
\usepackage{xcolor}

\newcommand{\wgr}[1]{\textcolor{black}{#1}}

\newcommand{\wred}[1]{\textcolor{black}{#1}}

\newcommand{\Div}{\mathrm{div}\,}
\newcommand{\B}{\Big}

\newcommand{\be}{\begin{equation}}
\newcommand{\ee}{\end{equation}}
\newcommand{\ba}{\begin{aligned}}
	\newcommand{\ea}{\end{aligned}}

\providecommand{\bysame}{\leavevmode\hbox to3em{\hrulefill}\thinspace}
\newcommand{\f}{\frac}

\newcommand{\ben}{\begin{enumerate}}
	\newcommand{\een}{\end{enumerate}}

\newcommand{\ti}{\nabla}

\newcommand{\Rmnum}[1]{\expandafter\@slowromancap\romannumeral #1@}

\allowdisplaybreaks

\numberwithin{equation}{section}
\begin{document}
\title{  $\varepsilon$-regularity criteria in anisotropic Lebesgue spaces   and Leray's self-similar solutions to
		   the 3D Navier-Stokes equations }
		\author{Yanqing Wang\footnote{ Department of Mathematics and Information Science, Zhengzhou University of Light Industry, Zhengzhou, Henan  450002,  P. R. China Email: wangyanqing20056@gmail.com},\;~Gang Wu\footnote{School of Mathematical Sciences,  University of Chinese Academy of Sciences, Beijing 100049, P. R. China Email: wugang2011@ucas.ac.cn}  ~~and ~\,
	Daoguo Zhou\footnote{
	College of Mathematics and Informatics, Henan Polytechnic University, Jiaozuo, Henan 454000, P. R. China Email:
	zhoudaoguo@gmail.com }}
\date{}
\maketitle
\begin{abstract}
In this paper,
we establish  some   $\varepsilon$-regularity criteria in anisotropic Lebesgue spaces    for suitable weak solutions to the 3D Navier-Stokes equations as follows:  \be\ba &\limsup\limits_{\varrho\rightarrow0}
 \varrho^{1-\f{2}{p}-\sum\limits^{3}_{j=1}\f{1}{q_{j}}}
\|u\|_{L_{t}^{p}L^{\overrightarrow{q}}_{x}(Q(\varrho))} \leq\varepsilon,
~~\f{2}{p}+\sum\limits^{3}_{j=1}\f{1}{q_{j}} \wred{\leq2}~~~~~\text{with}~q_{j} > 1;\\&
\label{wwwolf0}
\sup_{-1\leq t\leq0}\|u\|_{L^{\overrightarrow{q}}(B(1))} < \varepsilon,~~\f{1}{q_{1}}+\f{1}{q_{2}}+\f{1}{q_{3}}\wred{<2\quad \text{with}\, 1<q_{j}<\infty;}\\&\|u \|_{L_{t}^{p}L^{\overrightarrow{q}}_{x}(Q(1))} +\|\Pi\|_{L^{1}(Q(1))}\leq\varepsilon,
\wred{\quad \f2p+\sum^{3}_{j=1}\f{1}{q_{j}} <2 ~~~\text{with}~~  1<q_{j}<\infty,}
\ea
\ee
 which  extends   the  previous results  in \cite{[TX],[GKT],[HWZ],[GP],[Wolf1],[CV],[CKN]}.

As an application, in the spirit of \cite{[CW]},
 we prove that there
  does not exist a nontrivial Leray's backward self-similar solution with profiles in $L^{\overrightarrow{p}}(\mathbb{R}^{3})$ with $\f{1}{p_{1}}+\f{1}{p_{2}}+\f{1}{p_{3}}<2$.
  This  generalizes the corresponding  results of \cite{[NRS],[CW],[Tsai],[GP2]}.
  \end{abstract}
	\noindent {\bf MSC(2000):}\quad 35B65, 35D30, 76D05 \\\noindent
	{\bf Keywords:} Navier-Stokes equations;  suitable  weak solutions; regularity; self-similar solutions;anisotropic Lebesgue spaces
	\section{Introduction}
	\label{intro}
	\setcounter{section}{1}\setcounter{equation}{0}
		We study  the following   incompressible Navier-Stokes equations in   three-dimensional space
	\be\left\{\ba\label{NS}
	&u_{t} -\Delta  u+ u\cdot\ti
	u+\nabla \Pi=0, ~~\Div u=0,\\
	&u|_{t=0}=u_0,
	\ea\right.\ee
	where $u $ stands for the flow  velocity field, the scalar function $\Pi$ represents the   pressure.   The
	initial  velocity $u_0$ satisfies   $\text{div}\,u_0=0$.
	
We are concerned with the regularity of suitable weak solutions satisfying local energy inequality to the Navier-Stokes  system \eqref{NS}.
 A point $(x,t)$ is said to be a regular point if $|u| $ is bounded at some neighbourhood of this point. Otherwise, $(x,t)$ is singular  point.
 The local energy inequality of \eqref{NS} is due to Scheffer in  \cite{[Scheffer1],[Scheffer2]}. In this direction, a  milestones  result that one dimensional Hausdorff measure of the possible space-time singular \wgr{points set} of suitable weak solutions to the 3D
Navier-Stokes equations is zero was  obtained by
 Caffarelli, Kohn and Nirenberg   in \cite{[CKN]}. This result relies on the following two $\varepsilon$-regularity criteria in \cite{[CKN]} \wgr{for  suitable weak solutions to}  \eqref{NS}.
 One   holds at one scale: $(0,0)$ is regular point provided \begin{equation}\label{CKN}		\|u\|_{L^{3}(Q(1))}+\|u\Pi\|_{L^1(Q(1))}+\|\Pi\|_{L_{t}^{5/4}L_{x}^{ 1}(Q(1))}
		\leq \varepsilon.
		\end{equation}
The other needs infinitely many scales and an alternative assumption of \eqref{CKN} is that    \be\label{ckn2}
  \limsup_{\varrho\rightarrow0} \varrho^{-\f{1}{2}} \|\nabla u\|_{L^{2}(Q(\varrho))}
  \leq \varepsilon.
  \ee
An alternative  condition of \eqref{ckn2} is \wred{due to} Tian and Xin \cite{[TX]}
		 \be\label{tx}
   \limsup_{\varrho\rightarrow0} \varrho^{-\f{2}{3}} \|  u\|_{L^{3}(Q(\varrho))}
  \leq \varepsilon.
  \ee	
		  Gustafson,   Kang and  Tsai     \cite{[GKT]} enhanced \eqref{ckn2} and \eqref{tx} to the following results
\begin{align}\label{tsai1}
&\limsup_{\varrho\to 0 }\,\, \varrho^{1- \frac 2p- \frac 3 q}
\|u-\overline{u}_{\varrho}\|_{L_{t}^{p}L_{x}^{q}\wred{(Q(\varrho))}} \leq \varepsilon,\quad
1\leq 2/p +3/q\leq 2,\;  1\leq p, q \leq \infty;\\
&\limsup_{\varrho\to 0 }\,\, \varrho^{2-\frac 2p - \frac 3 q}
\|\nabla u\|_{L_{t}^{p}L_{x}^{q}}  \leq \varepsilon,\quad
2\le 2/p +3/q \le 3, \; 1 \le p,q \wred{\le \infty.}
\label{gkt2}
 \end{align}

Besides suitable weak solutions, there exists other kind of weak solutions equipping  energy inequality to the Navier-Stokes equations \eqref{NS}. This kind of weak solutions are called Leray-Hopf weak solutions.
 A number of papers have been devoted to the study of regularity of Leray-Hopf weak solutions    and many sufficient regularity conditions are established (see for example, \cite{[CMZ1],[CMZ2],[CMZ3],[HLLZ],[CZ],[CZZ],[KZ],[ZP],[NP],[KZ],[Wolf],[Qian],[Zheng],[GCS],[GKS],[CT]}). 
In particular, utilizing  the anisotropic Lebesgue spaces, Zheng first  studied anisotropic regularity criterion in terms of one velocity component in \cite{[Zheng]}. Later, Qian \cite{[Qian]}; Guo,   Caggio and   Skalak \cite{[GCS]};   Guo,   Kucera and   Skalak \cite{[GKS]}, further considered regularity condition in anisotropic Lebesgue spaces \wgr{for the Leary-Hopf weak solutions of} system  \eqref{NS}. It is worth pointing out that Sobolev-embedding theorem in anisotropic Lebesgue space was established in these works. For the details, see Lemma \ref{zc} in Section \ref{sec2}.
Inspired by recent works \cite{[Zheng],[Qian],[GKS],[GCS]}, we investigate $\varepsilon$-regularity criteria \wgr{for the} 3D Navier-Stokes equations  in anisotropic Lebesgue spaces.
Now we formulate our result as follows
\begin{theorem}
\label{thcw1}Let $(u,\,\Pi)$ be a suitable weak solutions to \eqref{NS} in $Q(\varrho)$. There exists a positive constant $\varepsilon_{1} $ \wred{such that if} \be\label{spe1}\limsup\limits_{\varrho\rightarrow0}
 \varrho^{1-\f{2}{p}-\sum\limits^{3}_{j=1}\f{1}{q_{j}}}
\|u-\overline{u}_{\varrho}\|_{L_{t}^{p}L^{\overrightarrow{q}}_{x}(Q(\varrho))} \wred{\leq\varepsilon_1},
~~\f{2}{p}+\sum\limits^{3}_{j=1}\f{1}{q_{j}} \leq2,~~~~~\text{with}~q_{j} > 1;
\ee
\wred{then} $(0,0)$ is a regular point.
\end{theorem}
\begin{remark}
\wred{The valid range of $q$} in \eqref{tsai1} is greater than or equal to $\f32$. Hence, we extend \eqref{ckn2}-\eqref{tsai1} to anisotropic Lebesgue spaces as well as the range of index of spatial integral.
\end{remark}
\wred{Combining} Theorem
\ref{thcw1} and
the absolute continuity of Lebesgues integral immediately imply the following result, which is of independent interest.
\begin{coro}\label{coro}
 Suppose that $(u,\,\Pi)$ is a suitable weak solution to \eqref{NS}. Then $(0,0)$ is a regular point provided one  of the following conditions holds
 \begin{align}
 (1).~~ &u\in L^{p}_{t}L^{\overrightarrow{q}}_{x}(Q(\varrho)), ~~~~ \text{with} ~~~ \f{2}{p}+\sum^{3}_{j=1}\f{1}{q_{j}}=1, ~~ 1< q_{j} ; \label{serin1}\\
 (2).~~ & u\in L^{\infty}_{t}L^{\overrightarrow{q}}_{x}(Q(\varrho)),~~
   \text{and} ~~~~ \|u\|_{\wred{L_t^{\infty}L_x^{\overrightarrow{q}}}(Q(\varrho))}\leq\varepsilon~~~~ \text{with} ~~~  \sum^{3}_{j=1}\f{1}{q_{j}}=1, ~~ \wred{1< q_{j}.} \label{serin2}
     \end{align}
\end{coro}
\begin{remark}
As in \cite{[GKT]},   one can examine that weak solutions satisfying \eqref{serin1} or \eqref{serin2} with $q_{j}>2$ \wred{are suitable weak solutions}. Therefore, this generalizes  Serrin's
classical work \cite{[Serrin]}.
\end{remark}
\begin{remark}
After we finished the main part of this   paper, we learnt that
an interesting  work involving well-posedness \wred{of \eqref{NS} with initial data} in $L^{\overrightarrow{q}}(\mathbb{R}^{3}) $  was established by Phan \cite{[Phan]}, where $\overrightarrow{q}$ meets $\sum\limits_{j=1}^3\f{1}{q_{j}}=1 $ with $q_{j}>2$.
\end{remark}

Next we \wgr{turn our attention} to the $\varepsilon$-regularity criteria at one scale in the type of \eqref{CKN}. In particular,
 Choi   and Vasseur \cite{[CV]},
		Guevara and Phuc \cite{[GP]} improved \eqref{CKN}	  to
		\begin{equation}\label{CVGP}
		\|u\|_{L_{t}^{ \infty}L_{x}^{2 }(Q(1))} + \|\nabla u\|_{L^{2}(Q(1))}  +\|\Pi\|_{L^{1}(Q(1))}   \leq \varepsilon.
		\end{equation}
Recently,
Guevara and Phuc \cite{[GP]} found that \eqref{CKN} can be replaced by the follows
		\be\label{GP}
		\|u\|_{L_{t}^{2p}L_{x}^{ 2q} (Q(1))}+\|\Pi\|_{L_{t}^{p}L_{x}^{ q}(Q(1))}\leq \varepsilon,
		~~~\f2p+\f3q=\f72
		~~~\text{with}~1\leq p\leq2.\ee
 Authors in \cite{[HWZ]} further   extended  \eqref{GP}  to
\be\label{opt}
\|u\|_{L_{t}^{p}L_{x}^{q}(Q(1))}+\|\Pi\|_{L^{1}(Q(1))}\leq\varepsilon,~~1\leq \f2p+\f3q <2, 1\leq p,\,q\leq\infty.
\ee
Very recently, an alternative proof of \eqref{opt} was presented  by Dong and Wang \cite{[DW]}. Moreover, for  a short summary  on    $\varepsilon$-regularity criteria at one scale we refer the reader to \cite{[HWZ]} and references therein.
The second result  in this paper  concerns $\varepsilon$-regularity criteria in anisotropic Lebesgue spaces at one scale, which generalizes the corresponding results in \eqref{opt}.
\begin{theorem}\label{anatones}
		Let  the pair $(u,  \Pi)$ be a suitable weak solution to the 3D Navier-Stokes system \eqref{NS} in $Q(1)$.
		There exists an absolute positive constant $\varepsilon$
		such that if the pair $(u,\Pi)$ satisfies	\begin{align}\label{anatonen}\|u &\|_{L_{t}^{p}L^{\overrightarrow{q}}_{x}(Q(1))} +\|\Pi\|_{L^{1}(Q(1))}\leq\varepsilon,\end{align}
\wgr{where}
\begin{align}&\f2p+\sum^{3}_{j=1}\f{1}{q_{j}} <2 ~~~\text{with}~~  1<q_{j}<\infty,\label{partragnge}\end{align}
		\wgr{then} $u\in L^{\infty}(Q(1/2)).$
	\end{theorem}
\begin{remark}
Theorem \ref{anatones}
extends the recent results obtained by Guevara and  Phuc in \cite{[GP]} and \cite{[HWZ]}.
\end{remark}

 Very recently, by means of $\varepsilon$-regularity criteria at one scale just in terms of  velocity filed,
 Chae and Wolf \cite{[CW]} considered   Leray's backward self-similar solutions to the Navier-Stokes equations with the \wred{profile in $L^{p}(\mathbb{R}^{3})$} ($p>\f32$). The
backward self-similar singular solutions to the Navier-Stokes was \wred{introduced by Leray}, who was the first to observe that one can construct singular solutions of the Navier-Stokes equations via
backward self-similar solutions.
The pair $(u,\Pi)$ \wred{is said} to be backward self-similar solutions if
$(u,\Pi)$ satisfy, for $a>0, T\in \mathbb{R},$
\be\ba
&u(x,t)=\wred{\f{1}{\sqrt{2a(T-t)}}U\bigg(\f{x}{\sqrt{2a(T-t)}}\bigg)}, \label{learysb}\\
&\Pi(x,t)=\wred{\f{1}{ 2a(T-t)}P\bigg(\f{x}{\sqrt{2a(T-t)}}\bigg)},
\ea\ee
where $U=(U_{1},U_{2},U_{3})$ and $P$ are defined in $\mathbb{R}^{3}$
and the pair $(u(x,t),\Pi(x,t))$ is defined in $\mathbb{R}^{3}\times(-\infty,T)$.
We obtain a singular solution at $t=T$ if \wred{$U\neq0$ and}

\be \label{SNS}
	-\Delta  U+a U+a(y\cdot \nabla)U+U\cdot\nabla U+ \nabla P=0  , ~~\Div U=0, ~~y\in \mathbb{R}^{3}.
	 \ee

In \wred{this direction}, \wred{the breakthrough} was made by Ne\v{c}as,  R\r{a}u\v{z}i\v{c}ka and   \v{S}ver\'{a}k in \cite{[NRS]}, where they prove that Leary's backward self-similar solutions is trivial under $U\in L^{3}(\mathbb{R}^{3})$.
Subsequently, Tsai \cite{[Tsai]} show that
the solution $U\in L^{p}( \mathbb{R}^{3})$ with $3<p<\infty$ in system \eqref{SNS} is zero
and
the solution $U\in L^{\infty}(\mathbb{R}^{3})$ in system \eqref{SNS} is constant.
Very recently, Guevara   and   Phuc \cite{[GP2]},Chae and Wolf \cite{[CW]} show that there does not exist a nontrivial solutions of \eqref{SNS} under more general assumptions. To the knowledge of the authors, we summarize the known results concerning Leary's backward self-similar solutions  with profiles in
isotropic Lebesgue spaces $L^{p}(\mathbb{R}^{3})$.
\begin{center}\begin{tabular}{|p{2cm}|p{2.5cm}|p{3.5cm}|p{2cm}|p{2cm}|}
\hline
\centering $\f{3}{2}<p$ & \centering $\f{12}{5}<p<6$& \centering$ p=3$ & \centering $3<p<\infty$ &   $p=\infty$ \\ \hline
\centering$U=0$&\centering$U=0$&\centering$U=0$&\centering$U=0$ &$U=cons$\\ \hline
Chae and  Wolf \cite{[CW]}& Guevara and  Phuc \cite{[GP2]}& Ne\v{c}as,  R\r{a}u\v{z}i\v{c}ka and   \v{S}ver\'{a}k \cite{[NRS]} & Tsai  \cite{[Tsai]}  & Tsai \cite{[Tsai]}
\\ \hline
\end{tabular}\end{center}

We shall investigate Leray's backward self-similar solutions to the Navier-Stokes equations with the profile in anisotropic Lebesgue spaces. More precisely, we have the following statement.
 \begin{theorem}\label{[Leraybackself]}
 Let the pair $(U,P) \in C^{\infty}(\mathbb{R}^{3})^{3}\times C^{\infty}(\mathbb{R}^{3})$ be a solutions to \eqref{SNS}. Assume that, for $ 1<p_{j}<\infty $,
   \be\label{[Leraybackselfr]}
    U\in L^{\overrightarrow{p}}(\mathbb{R}^{3})~~
    \text{with}~~ \f{1}{p_{1}}+\f{1}{p_{2}}+\f{1}{p_{3}}<2,\ee
  then $U=0$.
 \end{theorem}
 \begin{remark}
A special case of (\ref{[Leraybackselfr]}) is $U\in L_{1}^{3/2}L_{2}^{3/2}L_{3}^{p_{3}}(\mathbb{R}^{3})$ \wred{with $p_3>\f32$}, which is more general than that in \cite{[CW]}.
\end{remark}
\begin{remark}
We would like to mention that Phan recently considered
 Liouville type theorems for 3D stationary Navier-Stokes equations in weighted mixed-norm Lebesgue
spaces.
\end{remark}
The proof of Theorem  \ref{[Leraybackself]} is based on Chae and Wolf's approach \cite{[CW]}. We present the framework of the argument utilized in \cite{[CW]}.
\begin{enumerate}[-]
\item
First, Chae and Wolf derived the $\varepsilon$-regularity criteria at one scale without the pressure \wred{as below}
\be \label{wolfchae}
\sup_{-1\leq t\leq0}\|u\|_{L^{q}(B(1))} < \varepsilon,~~\f32<q.
\ee
\item
Second, it follows from \cite[Theorem 1,   p.31 ]{[Tsai]} that $U$ belonging
to $L^{\infty}(\mathbb{ R}^{3})\bigcap L^{p}(\mathbb{ R}^{3}) $ with $1\leq p<\infty$ is trivial. Therefore, the key point in this step is to show that   $U\in L^{\infty}(\mathbb{ R}^{3})$. It is well-known that $u\in C^{\infty}(\mathbb{R}^{3})$. As a consequence, \wred{it suffices} to prove that the bound of $U(x)$ for $|x|\geq R$ with
suffice large $R$. The decay at infinity of $L^{p}(\mathbb{R}^{3})$ together with the $\varepsilon$-regularity criteria \eqref{wolfchae} yields the desired estimate.
\end{enumerate}
Based on this, we see that it is enough to generalize
\eqref{wolfchae} to extend Leray's self-similar solutions. Therefore,  Theorem  \ref{[Leraybackself]} turns out to be a  consequence of the following result.
 \begin{theorem}\label{the1.5}
		Let  the pair $(u,  \Pi)$ be a suitable weak solution to the 3D Navier-Stokes system \eqref{NS} in $Q(1)$.
		There exists an absolute positive constant $\varepsilon$
		such that if   $u$ satisfies	
\be \label{wwwolf}
\sup_{-1\leq t\leq0}\|u\|_{L^{\overrightarrow{q}}(B(1))} < \varepsilon,~~\f{1}{q_{1}}+\f{1}{q_{2}}+\f{1}{q_{3}}<2, 1<q_{j}<\infty.
\ee
		then, $u\in L^{\infty}(Q(1/16)).$
	\end{theorem}
The proof of \eqref{wolfchae} and \eqref{wwwolf} relies on the local suitable weak solutions via  local pressure projection. The   local pressure projection and local suitable weak solutions is due to Wolf \cite{[Wolf1]}. The    novelty in this concept is that the local energy
inequality
\eqref{wloc1} removed   the non-local effect of the pressure term.
It is shown that any suitable weak solutions is local suitable weak solutions in \cite{[CW17]}. We refer the reader to \cite{[JWZ]} and \cite{[WWZ]} for other $\varepsilon$-regularity criteria at one scale without pressure.

This  paper is organized as follows. In the  second section,
we present the notations and some known fact such as  interpolation inequality and Sobolev embedding theorem  in anisotropic Lebesgue spaces. In addition, we also recall
 the local suitable weak solutions.
 Section 3 is devoted to
proving Theorem \ref{thcw1}. In Section 4,   we prove Theorems \ref{anatones} involving $\varepsilon$-regularity criteria at one scale. In \wred{final section},   we complete the proof of Theorems \ref{the1.5} concerning  $\varepsilon$-regularity criteria at one scale just via velocity filed. This also means Theorem \ref{[Leraybackself]}.

 \section{Notations and \wred{some known facts}} \label{sec2}

A function $f$ belongs to  the anisotropic Lebesgue spaces $L^{\overrightarrow{q}}_{x}(\Omega)$
if
  $$\|f\|_{L^{\overrightarrow{q}}_{x}(\Omega)}=\|f\|_{L_{1}^{q_{1}}L_{2}^{q_{2}}L_{3}^{q_{3}}(\Omega)}=
  \B\|\big\|\|u\|_{L_{1}^{q_{1}}(\{x_1:x\in\Omega\})}\big\|_{L_{2}^{q_{2}}(\{{x_2:x\in\Omega}\})}\B\|_{L_{3}^{q_{3}} (\{x_3:x\in\Omega\})}<\infty.  $$
The study of
 anisotropic Lebesgue spaces  first appears in
  Benedek and   Panzone \cite{[BP]}.
 \wred{Various topics}  on mixed Lebesgue spaces
  were established (see e.g. \cite{[Ward],[RRT],[BP]} \wred{and references therein}).

For $p\in [1,\,\infty]$, the notation $L^{p}((0,\,T);X)$ stands for the set of measurable functions on the interval $(0,\,T)$ with values in $X$ and $\|f(t,\cdot)\|_{X}$ belongs to $L^{p}(0,\,T)$.

   For simplicity,   we write $$\|f\| _{L_{t}^{p}L_{x}^{\overrightarrow{q}}(Q(\varrho))}:=\|f\| _{L^{p}(-\varrho^{2},0;L^{\overrightarrow{q}}(B(\varrho)))}, $$
  where $Q(\varrho)=B(\varrho)\times ( -\varrho^{2}, 0)$ and $ B(\varrho)$ denotes the ball of center $0$ and radius $\varrho$.

  In what follows, for the sake of simplicity of presentation, we define $$\wred{\f{1}{\overrightarrow{q}}=\frac{1}{q_1}+\frac{1}{q_{2}}+\frac{1}{q_{3}}}. $$
We denote
  the average of $f$ on the ball $B(r)$ by
$\overline{f}_{r}$.

 Moreover, for the convenience of the reader, we state a fact which will be   frequently used below
\be\label{simplefact}
\wgr{\Omega_{1}\subseteq\Omega_{2}\subseteq\Omega_{3},}
\ee
where
$$\Omega_{1}=\{x:|x|<1\},\Omega_{2}=\{x:|x_1|,|x_2|,|x_3|<1\},
\Omega_{3}=\{x:|x|<\sqrt{3}\}.
$$
   The classical Sobolev space $W^{k,p}(\Omega)$ is equipped with the norm $\|f\|_{W^{k,p}(\Omega)}=\sum\limits_{\alpha =0}^{k}\|D^{\alpha}f\|_{L^{p}(\Omega)}$. We denote by  $ \dot{H}^{s}$ homogeneous Sobolev spaces with the norm $\|f\|^{2} _{\dot{H}^{s}}= \int_{\mathbb{R}^{3}} |\xi|^{2s}|\hat{f}(\xi)|^{2}d\xi$.
 We will also use  the summation convention on repeated indices.
 $C$ is an absolute
   constant which may be different from line to line unless otherwise stated.

Now,    for the convenience of readers, we recall the \wred{classical definition of   suitable weak solution} to the Navier-Stokes system \eqref{NS}.
	\begin{definition}\label{defi1}
		A  pair   $(u, \,\Pi)$  is called a suitable weak solution to the Navier-Stokes equations \eqref{NS} provided the following conditions are satisfied,
		\begin{enumerate}[(1)]
			\item $u \in L^{\infty}(-T,\,0;\,L^{2}(\mathbb{R}^{3}))\cap L^{2}(-T,\,0;\,\dot{H}^{1}(\mathbb{R}^{3})),\,\Pi\in
			L^{3/2}(-T,\,0;L^{3/2}(\mathbb{R}^{3}));$
			\item$(u, ~\Pi)$~solves (\ref{NS}) in $\mathbb{R}^{3}\times (-T,\,0) $ in the sense of distributions;
			\item$(u, ~\Pi)$ satisfies the following inequality, for a.e. $t\in[-T,0]$,
			\begin{align}
				&\int_{\mathbb{R}^{3}} |u(x,t)|^{2} \phi(x,t) dx
				+2\int^{t}_{-T}\int_{\mathbb{R} ^{3 }}
				|\nabla u|^{2}\phi  dxds\nonumber\\ \leq&  \int^{t}_{-T }\int_{\mathbb{R}^{3}} |u|^{2}
				(\partial_{s}\phi+\Delta \phi)dxds
				+ \int^{t}_{-T }
				\int_{\mathbb{R}^{3}}u\cdot\nabla\phi (|u|^{2} +2\Pi)dxds, \label{loc}
			\end{align}
			where non-negative function $\phi(x,s)\in C_{0}^{\infty}(\mathbb{R}^{3}\times (-T,0) )$.\label{SWS3}
		\end{enumerate}
	\end{definition}

\begin{lemma}{\cite{[BP]}}\label{interi}
 Suppose $\Omega\subset\mathbb{R}^{3}$ and $1\leq s_{j}\leq r\leq t_{j}\leq\infty$ and
$$
\f{1}{r}=\f{\theta}{s_{j}}+\f{1-\theta}{t_{j}}.
$$
Assume also $f\in L^{\overrightarrow{s}}(\Omega)\cap L^{\overrightarrow{t}}(\Omega)$. Then
 $ f\in L^{\overrightarrow{r}}(\Omega)$ and
\be\label{interinequality}
\|f\|_{L^{r}(\Omega) }
\leq \|f\|^{\theta}_{L^{\overrightarrow{s}}(\Omega) }\|f\|^{1-\theta}_{L^{\overrightarrow{t}}(\Omega) }.
\ee
\end{lemma}
\begin{proof}
By the H\"older inequality, we know that
$$\ba
\|f\|_{L^{r}(\Omega)}&=\||f|^{r\theta}|f|^{r(1-\theta)}\|^{\f{1}{r}}_{L^{1}(\Omega)}
\\&\leq \||f|^{r\theta}\|^{\f{1}{r}}_{L^{\f{\overrightarrow{s}}{r\theta}}(\Omega)} \||f|^{r(1-\theta)}\|^{\f{1}{r}}_{L^{\f{\overrightarrow{t}}{r(1-\theta)}}(\Omega)}
\\&\leq
\|f\|^{\theta}_{L^{\overrightarrow{s}}(\Omega) }\|f\|^{1-\theta}_{L^{\overrightarrow{t}}(\Omega) }.
\ea$$
\end{proof}

We recall the  Sobolev
embedding theorem in anisotropic Lebesgue space in the full three-dimensional space. We refer to \cite{[GCS]} for the proof of the following result.
\begin{lemma}{\cite{[Zheng],[GCS],[Qian],[GKS]}}\label{zc} Let $q_{1},q_{2},q_{3}\in[2,\infty)$ and \wred{$0\leq \f{1}{\overrightarrow{q}}-\f12\leq 1$}. Then there exists a constant $C$ such that
\begin{align}\label{anis}\|f\|_{\wred{L^{\overrightarrow{q}}(\mathbb{R}^3)}}
 &\leq C
\|\partial_{1}f\|^{\f{q_{1}-2}{2q_{1}}}_{L^{2}(\mathbb{R}^{3})}\|\partial_{2}f\|^{\f{q_{2}-2}{2q_{2}}}_{L^{2}(\mathbb{R}^{3})}
\|\partial_{3}f\|^{\f{q_{3}-2}{2q_{3}}}_{L^{2}(\mathbb{R}^{3})}\|f\|^{\f{1}{\overrightarrow{q}}-\f12}_{L^{2}(\mathbb{R}^{3})}
\nonumber\\&\leq C\|\nabla f\|_{L^{2}(\mathbb{R}^{3})}^{\f32-{\f{1}{\overrightarrow{q}}}}\|f\|^{\f{1}{\overrightarrow{q}}-\f12}_{L^{2}(\mathbb{R}^{3})}.
\end{align}
\end{lemma}
\wred{We can state} the local version of the above lemma.
\begin{lemma}\label{zcl} \wred{Let $q_{1},q_{2},q_{3}\in[2,\infty)$ and $0\leq \f{1}{\overrightarrow{q}}-\f12\leq 1$.} Then, for $\varrho>0$ and $0<\xi<\eta$, there exists a constant $C$ such that
\begin{align}
&\|f\|_{L_{1}^{q_{1}} L_{2}^{q_{2}} L_{3}^{q_{3}}(B(\varrho) )} \leq C\|\nabla f\|_{L^{2}(B(\sqrt{2}\varrho ))}^{\f32-\f{1}{\overrightarrow{q}}}\|f\|^{\f{1}{\overrightarrow{q}}-\f{1}{2}}_{L^{2}(B(\sqrt{2}\varrho))} +C\varrho^{-({\f32-\f{1}{\overrightarrow{q}}})}\|f\| _{L^{2}(B(\sqrt{2}\varrho ))},
\label{locan}\\
&\|f\|_{L_{1}^{q_{1}} L_{2}^{q_{2}} L_{3}^{q_{3}}(B(\f{\xi+3\eta}{4}) )} \leq C\|\nabla f\|_{L^{2}(B( \eta))}^{ \f32-\f{1}{\overrightarrow{q}}} \|f\|^{\f{1}{\overrightarrow{q}}-\f{1}{2}}_{L^{2}(B( \eta))} + C(\eta-\xi)^{- ({\f32-\f{1}{\overrightarrow{q}}})}\|f\| _{L^{2}(B( \eta))}.\label{locanlast}
\end{align}
\end{lemma}
\begin{proof}
  Let $\phi(x)$ be non-negative smooth function supported in $B(\sqrt{2}\varrho)$ such that
$\phi(x)\equiv1$ on $B(\varrho )$, $0\leq\phi(x)\leq1$ and $|\nabla \phi| \leq  C/\varrho $.\\
Making use of\eqref{anis}, we see that
$$\ba\|f\|_{L_{1}^{q_{1}} L_{2}^{q_{2}} L_{3}^{q_{3}}(B(\varrho))}\leq&
\|f\phi\|_{L_{1}^{q_{1}} L_{2}^{q_{2}} L_{3}^{q_{3}}(\mathbb{R}^{3})}\\
\leq& C\|\nabla (f\phi)\|_{L^{2}(\mathbb{R}^{3})}^{\f32-\f{1}{\overrightarrow{q}}}\|f\phi\|^{\f{1}{\overrightarrow{q}}-\f{1}{2}}_{L^{2}(\mathbb{R}^{3})} \\\leq& C\B(\|\phi\nabla f\|_{L^{2}(\mathbb{R}^{3})}+\|f\nabla \phi\|_{L^{2}(\mathbb{R}^{3})}\B)^{\f32-\f{1}{\overrightarrow{q}}}\|f\phi\|^{\f{1}{\overrightarrow{q}}-\f{1}{2}}_{L^{2}(\mathbb{R}^{3})} \\\leq& C\|\nabla f\|_{L^{2}(B(\sqrt{2}\varrho))}^{\f32-\f{1}{\overrightarrow{q}}}\|f \|^{\f{1}{\overrightarrow{q}}-\f{1}{2}}_{L^{2}(B(\sqrt{2}\varrho))} +C\varrho^{-(\f32-\f{1}{\overrightarrow{q}})}\|f\| _{L^{2}(B(\sqrt{2}\varrho))},
\ea$$
which means \eqref{locan}.\\
Along the exact same lines as the above proof, we have \eqref{locanlast}.
This achieves the proof of the desired estimate.
\end{proof}
By the Poincar\'e inequality, we know that
$$\|f-\wred{\bar{f}_{B(\sqrt{2}\varrho)}}\|_{L^{2}(B(\sqrt{2}\varrho))}\leq
C\rho\|\nabla f\|_{L^{2}(B(\sqrt{2}\varrho))}.
$$
This allows us to derive from  \eqref{locan} that, for any $\int_{B(\sqrt{2}\varrho)}fdx=0$, \wred{$0\leq \f{1}{\overrightarrow{q}}-\f12\leq 1$,}
\be
\wred{\|f\|_{L_{1}^{q_{1} } L_{2}^{q_{2} } L_{3}^{_{q_{3} }}(B(\varrho) )}    \leq C\|\nabla f\|_{L^{2}(B(\sqrt{2}\varrho ))}^{\f32-\f{1}{\overrightarrow{q}}}\|f\|^{\f{1}{\overrightarrow{q}}-\f12}_{L^{2}(B(\sqrt{2}\varrho))},}
\label{locan0}\ee
which means that
$$\ba
\|f\|_{L^{m}_{t}L_{1}^{q_{1} } L_{2}^{q_{2} } L_{3}^{_{q_{3} }}(Q(\varrho) )}    &\leq C\|\nabla f\|_{L^{2}(Q(\sqrt{2}\varrho ))}^{\f32-\f{1}{\overrightarrow{q}}}\|f\|^{\f{1}{\overrightarrow{q}}-\f12}_{L^{\infty}_{t}L^{2}(B(\sqrt{2}\varrho))}
\\&\leq C\|\nabla f\|_{L^{2}(Q(\sqrt{2}\varrho ))}^{\f{2}{m}}\|f\|^{1-\f{2}{m}}_{L^{\infty}_{t}L^{2}(B(\sqrt{2}\varrho))},
\ea$$
where
$$
\f{2}{m}+\f{1}{\overrightarrow{q}}=\f32.
$$

Next, we state another lemma on decomposition of pressure obtained in \cite{[HWZ]} that will be used in the proof of Theorem \ref{anatones}.
 		\begin{lemma}\cite{[HWZ]}\label{lem2}
Let $\Phi$ denote  the standard normalized fundamental solution of Laplace equation in $\mathbb{R}^{3}$. For $0<\xi<\eta$, we consider smooth cut-off function $\psi\in C^{\infty}_{0}(B(\f{\xi+3\eta}{4}))$ such that $0\leq\psi\leq1$ in $B(\eta)$, $\psi\equiv1$ in $B(\f{3\xi+5\eta}{8})$  and $|\nabla^{k}\psi |\leq C/(\eta-\xi)^{k}$ with $k=1,2$ in \wgr{$B(\eta)$}.  Then we may split  pressure $\Pi$ in \eqref{NS} \wgr{as below}
			\be\label{decompose pk}
			\Pi(x):=\Pi_{1}(x)+\Pi_{2}(x)+\Pi_{3}(x), \quad x\in B(\f{\xi+\eta}{2}),
			\ee
			where
			$$\ba
			\Pi_{1}(x)=&-\partial_{i}\partial_{j}\wgr{\Phi} \ast (\psi (u_{j}u_{i})) ,\\
			\Pi_{2}(x)
			=&2\partial_{i}\wgr{\Phi} \ast(\partial_{j}\psi(u_{j}u_{i}))- \wgr{\Phi} \ast
			(\partial_{i}\partial_{j}\psi u_{j}u_{i}), \\
			\Pi_{3}(x)
			=&2\partial_{i}\wgr{\Phi} \ast(\partial_{i}\psi \Pi)-\wgr{\Phi} \ast(\partial_{i}\partial_{i}\psi \Pi).
			\ea
			$$
 Moreover, there holds
			\begin{align}
				& \| \Pi_1\|_{L^{3/2}(Q(\f{\xi+\eta}{2}))}
				\leq C\|  u\|^{2}_{L^{3}(Q(\f{\xi+3\eta}{4}))};\label{p1estimate}\\
				& \|  \Pi_2\|_{L^{3/2}(Q(\f{\xi+\eta}{2}))}
				\leq  \wgr{\f{C\eta^{3}}{(\eta-\xi)^{3}}}\|  u\|^{2}_{L^{3}(Q(\f{\xi+3\eta}{4}))};\label{p2estimate}\\
				& \|  \Pi_3\|_{\wgr{L^1L^{2}}(Q(\f{\xi+\eta}{2}))}
				\leq \wgr{\f{C\eta^{3/2 }}{(\eta-\xi)^{3}}}\|\Pi\|_{L^{1}(Q(\f{\xi+3\eta}{4}))}.\label{p3estimate}
			\end{align}
		\end{lemma}
\wred{Now we introduce Wolf's local} pressure projection $\mathcal{W}_{p,\Omega}:$ $W^{-1,p}(\Omega)\rightarrow W^{-1,p}(\Omega)$ $(1<p<\infty)$.
 More precisely, for any  $f\in W^{-1,p}(\Omega)$, we define \wred{$\mathcal{W}_{p,\Omega}(f)= \nabla\Pi$}, where $\Pi$ satisfies \eqref{GMS}.
Let $\Omega$  be a  bounded domain with $\partial\Omega\in C^{1}$.
According to the $L^p$ theorem of Stokes system in \cite[Theorem 2.1, p149]{[GSS]},
there exists a unique pair $(b,\Pi)\in W^{1,p}(\Omega)\times L^{p}(\Omega)$ such that
\be\label{GMS}
-\Delta b+\nabla\Pi=f,~~ \text{div}\,b=0, ~~b|_{\partial\Omega}=0,~~ \int_{\Omega}\Pi dx=0.
\ee
Moreover, this pair is subject to the inequality
$$
\|b\|_{\wred{W^{1,p}}(\Omega)}+\|\Pi\|_{\wred{L^p}(\Omega)}\leq C\|f\|_{\wred{W^{-1,p}}(\Omega)}.
$$
Let $\nabla\Pi= \mathcal{W}_{p,\Omega}(f)$ $(f\in L^p(\Omega))$, then $\|  \Pi\|_{L^p(\Omega)}\leq C\|f\|_{L^p(\Omega)},$ where we used the fact that $L^{p}(\Omega)\hookrightarrow W^{-1,p}(\Omega)$.  Moreover, from $\Delta \Pi=\text{div}\,f$, we see that $\|  \nabla\Pi\|_{L^p(\Omega)}\leq C(\|f\|_{L^p(\Omega)}+ \|  \Pi\|_{L^p(\Omega)}) \leq C\|f\|_{L^p(\Omega)}.$
\wred{For any} ball $B(R)\subseteq \mathbb{R}^{3}$, by the local pressure projection,
 Wolf et al. presented the pressure decomposition
 $$
- \nabla \Pi = - \partial _ { t } \nabla \Pi _ { h } - \nabla \Pi_ { 1 } - \nabla \Pi_ { 2 },
$$
where
   $$\nabla\Pi_{h}=-\mathcal{W}_{p,B(R)}(u),~~ \nabla\Pi_{1}=\mathcal{W}_{p,B(R)}(\Delta u),~~\nabla\Pi_{2}=-\mathcal{W}_{p,B(R)}( u\cdot\nabla u).$$ After denoting     $v=u+\nabla\Pi_{h}$, one gets the local
     energy inequality,
for a.e. $t\in[-T,0]$ and non-negative function $\phi(x,s)\in C_{0}^{\infty}(\mathbb{R}^{3}\times (-T,0) )$,
			 \begin{align}
  &\int_{B(r)}|v|^2\phi (x,t)  d  x+ \int^{t}_{-T }\int_{B(r)}\big|\nabla v\big|^2\wred{\phi (x,s)} d  x ds  \nonumber\\  \leq&   \int^{t}_{-T }\int_{B(r)} | v |^2(  \Delta \phi +  \partial_{t}\phi )  d  x d s +\int^{t}_{-T }\int_{B(r)}|v|^{2}u\cdot\nabla \phi    \wred{dxds}\nonumber\\
& +\int^{t}_{-T }\int_{B(r)} \phi ( u\otimes v :\nabla^{2}\Pi_{h} ) \wred{dxds}   +\int^{t}_{-T }\int_{B(r)}  \Pi_{1}v\cdot\nabla \phi   dxds+\int^{t}_{-T }\int_{B(r)}  \Pi_{2}v\cdot\nabla \phi   dxds.\label{wloc1}
 \end{align}
With this in hand, we  present the Wolf's new definition of suitable weak solutions of Navier-Stokes equations \eqref{NS}.
	\begin{definition}\label{defi}
		A  pair   $(u, \,\Pi)$  is called a suitable weak solution to the Navier-Stokes equations \eqref{NS} provided the following conditions are satisfied,
		\begin{enumerate}[(1)]
			\item $u \in L^{\infty}(-T,\,0;\,L^{2}(\mathbb{R}^{3}))\cap L^{2}(-T,\,0;\,\dot{H}^{1}(\mathbb{R}^{3})),\,\Pi\in
			L^{3/2}(-T,\,0;L^{3/2}(\mathbb{R}^{3}));$\label{SWS1}
			\item$(u, ~\Pi)$~solves (\ref{NS}) in $\mathbb{R}^{3}\times (-T,\,0) $ in the sense of distributions;\label{SWS2}
			\item The local energy inequality \eqref{wloc1} is valid and $\nabla \Pi_{h}$ is a  harmonic function. In addition, $ \nabla\Pi_{h}, \nabla\Pi_{1}$ and $\nabla\Pi_{2}$ meet the \wred{following facts}
	\begin{align}   &\|\nabla\Pi_{h}\|_{L^p(B(R))}\leq  \|u\|_{L^p(B(R))}, \label{ph}\\
 &\|\wred{\Pi_{1}}\|_{L^2(B(R))}\leq  \|\nabla u\|_{L^2(B(R))},\label{p1}\\
 &\|\wred{\Pi_{2}}\|_{L^{p/2}(B(R))}\leq  \| |u|^{2}\|_{L^{p/2}(B(R))}.\label{p2}
\end{align}	\end{enumerate}
	\end{definition}
	\noindent
We list some
interior estimates
of  harmonic functions $\Delta h=0$, which will be frequently utilized later. Let $1\leq p,q\leq\infty$ and $0<r<\rho$, then, it holds
\be\label{h1}\|\nabla^{k}h\|_{L^{q}
(B(r))}\leq \f{Cr^{\f{n}{q}}}{(\rho-r)^{\f{n}{p}+k}}\|h\|_{L^{p}(B(\rho))}.\ee
\be\label{h2}
 \| h-\overline{h}_{r}\|_{L^{q}
(B(r))}\leq \f{Cr^{\f{n}{q}+1}}{(\rho-r)^{\f{n}{q}+1 }}\|h-\overline{h}_{\rho}\|_{L^{q}(B(\rho))}.\ee

\section{Regularity criteria in anisotropic Lebesgue space at infinitely many scales}
Inspired by   \cite{[GKT]}, we present the proof of Theorem \ref{thcw1} by Lemma \ref{ineq} and Lemma \ref{presure}.

\wred{By
the natural scaling} property of Navier-Stoke equations \eqref{NS}, we  introduce the following dimensionless quantities,
\begin{align}
&E_{\ast}(\varrho)=\frac{1}{\varrho}\iint_{\wgr{Q(\varrho)}}|\nabla u|^2dx dt,& E(\varrho)=\sup_{-\varrho^2\leq   t<0}\frac{1}{\varrho}\int_{B(\varrho)}|u|^2dx,\nonumber\\
&E_{p}(\varrho)=\frac{1}{r^{5-p}}\iint_{\wgr{Q(\varrho)}}|u|^{p}dx dt,&D_{3/2}(\varrho)=\frac{1}{r^{2}}\iint_{\wgr{Q(\varrho)}}|\Pi-\bar{\Pi}_{\wgr{B(\varrho)}}|^{\frac{3}{2}}dx
dt. \nonumber
\end{align}
 According to  the H\"older inequality, it suffices   to prove Theorem \ref{thcw1} for the case
 $$
 \wred{\f{2}{p}+\sum\limits_{j=1}\limits^{3}\frac{1}{q_{j}}=2.}
	 $$
Therefore, we introduce the     dimensionless quantities below
\begin{align}
  \wred{E_{p,\overrightarrow{q}}(u,\,\varrho)=
\varrho^{-1}\|u-\overline{u}_{\varrho}\|_{L^{p}_tL^{\overrightarrow{q}}_{x}(Q(\varrho))}.}
\nonumber\end{align}
To prove Theorem \ref{thcw1}, we need the  following crucial lemma.

\begin{lemma}\label{ineq}
For $0<\sqrt{6}\mu\leq \rho$,~
there is an absolute constant $C$  independent of  $\mu$ and $\rho$,~ such that
 \begin{align}
\wred{E_{3}(\mu)\leq C\B(\f{\rho}{\mu}\B)^{2}E_{p,\overrightarrow{q}} (u,\,\rho) E_{\ast}(\rho)^{1-\f{1}{p}} E^{\f{1}{p}}(\rho) +C\B(\f{\mu}{\rho}\B)E_{3}(\rho).}
\label{ineq2/2}    \end{align}
\end{lemma}
\begin{proof}
 We begin by proving the following crucial estimate
\be\ba\label{key2.9} \iint_{Q(\varrho)} |v|^{3}dxds&=\iint_{Q(\varrho)} |v| |v|^{2}dxds
\\& \leq C \|v\|_{\wred{L_{t}^{p}L_x^{\overrightarrow{q}}}(Q(\varrho))} \|v\|^{2}_{\wred{L_{t}^{2p^{\ast}}L_x^{2\overrightarrow{q}^{\ast}}}(Q(\varrho))}
\\& \leq C \|v\|_{\wred{L_{t}^{p}L_x^{\overrightarrow{q}}}(Q(\varrho))} \|\nabla v\|_{L^{2}(Q(\sqrt{2}\varrho ))}^{\f{2}{p^{\ast}}}\|v\|^{2-\f{2}{p^{\ast}}}_{L^{\infty}_{t}\wred{L_x^{2}(Q(\sqrt{2}\varrho))}},
\\& \leq C \|v\|_{\wred{L_{t}^{p}L_x^{\overrightarrow{q}}}(Q(\varrho))} \|\nabla v\|_{L^{2}(Q(\sqrt{2}\varrho ))}^{2-\f{2}{p}}\|v\|^{\f{2}{p}}_{L^{\infty}_{t}\wred{L_x^{2}(Q(\sqrt{2}\varrho))}},
\ea\ee
where $v=u-\bar{u}_{\sqrt{2}\varrho}$.

\wred{Noticing that} $\f{2}{p}+\f{1}{\overrightarrow{q}}=2$, we infer that
\be\ba\label{lem2.3.2.7}
\iint_{\wgr{Q(\varrho)}}|v|^{3}dxdt & \leq C\|v\|_{L_{t}^{p}L^{\overrightarrow{q}}_{x}\wgr{(Q(\varrho))}}  \|\nabla v\|_{L^{2}(Q(\sqrt{2}\varrho))}^{ \f{1}{\overrightarrow{q}} } \wred{\|v\|^{2-\f{1}{\overrightarrow{q}}}_{L_t^{\infty}L_x^{2}\wgr{(Q(\sqrt{2}\varrho))}},}
\ea\ee
\wred{which entails}  that
\be\ba
&\iint_{Q( \varrho)}|u- \overline{u}_{\sqrt{2}\varrho}|^{3}dxdt \\
\leq& C\|u- \overline{ u}_{\sqrt{2}\varrho }\|_{L_{t}^{p}L^{\overrightarrow{q}}_{x}(Q(\varrho))}   \|\nabla u\|_{L^{2}(\wgr{Q(\sqrt{2}\varrho)})}^{ \f{1}{\overrightarrow{q}} } \|u- \overline{ u}_{\sqrt{2}\varrho} \|^{2-\f{1}{\overrightarrow{q}}}_{\wgr{L^{\infty}L^{2}}(Q(\sqrt{2}\varrho))}.\label{eq3.4}
\ea\ee
  By virtue of  the triangle inequality, \wred{we obtain}
\begin{align}\nonumber
\iint_{Q(\mu)}|u|^{3}dx\leq& C\iint_{Q(\mu)}|u-\bar{u}_{{\rho}}|^{3}dx
+C\iint_{\wred{Q(\mu)}}|\bar{u}_{{\rho}}|^{3} dx\\
\leq& C\iint_{Q(\rho/\sqrt{6})}|u-\bar{u}_{{\rho_{}}}|^{3}dx
 +
 C\f{\mu^{3}}{\rho^{3}}\B( \iint_{\wred{Q(\rho)}}|u|^{3}dx\B). \label{lem2.31}
 \end{align}
 This \wred{together with \eqref{eq3.4}} implies the desired estimate \eqref{ineq2/2}.
 \end{proof}\begin{lemma}\label{presure}
For $0<4\sqrt{6}\mu\leq \rho$, there exists an absolute constant $C$  independent of $\mu$ and $\rho$ such that
\begin{align}
&D_{3/2}(\mu)\leq
C\left(\f{\rho}{\mu}\right)
^{2}\wgr{E_{p,\overrightarrow{q} }}(\rho)\wred{E_{\ast}(\rho)^{1-\f{1}{p}} E^{\f{1}{p}}(\rho)}
+C\left(\f{\mu}{\rho}\right)^{\f{5}{2}}D_{3/2}(\rho).\label{pe}
\end{align}
\end{lemma}
\begin{proof}
We consider the usual cut-off function $\phi\in C^{\infty}_{0}(B(\f{\rho}{ \sqrt{6}}))$ such that $\phi\equiv1$ on $B(\f{3\rho}{4\sqrt{6}})$ with $0\leq\phi\leq1$,
$|\nabla\phi |\leq C\rho^{-1} $ and $ ~|\nabla^{2}\phi |\leq
C\rho^{-2}.$\\
\wgr{Due to} the incompressible condition, the pressure equation can be \wgr{written as}
$$
\partial_{i}\partial_{i}(\Pi\phi)=-\phi \partial_{i}\partial_{j} U_{i,j}
+2\partial_{i}\phi\partial_{i}\Pi+\Pi\partial_{i}\partial_{i}\phi
,$$
where $U_{i,j}=(u_{j}- \bar{u}_{{\rho_{/\sqrt{6}}}})(u_{i}-\bar{u}_{{\rho_{/\sqrt{6}}}})$. \wgr{Thus it follows} that, for $x\in B(\f{3\rho}{4\sqrt{6}})$
\be\ba\label{pp}
\Pi(x)=&\Phi\ast \{-\phi \partial_{i}\partial_{j} U_{i,j}
+2\partial_{i}\phi\partial_{i}\wgr{\Pi}+\wgr{\Pi}\partial_{i}\partial_{i}\phi
\}\\
=&-\partial_{i}\partial_{j}\Phi \ast (\phi  U_{i,j} )\\
&+2\partial_{i}\Phi \ast(\partial_{j}\phi U_{i,j} )-\Phi \ast
(\partial_{i}\partial_{j}\phi U_{i,j} )\\
& \wgr{+2}\partial_{i}\Phi \ast(\partial_{i}\phi \Pi) -\Phi \ast(\partial_{i}\partial_{i}\phi \Pi)\\
\wgr{=:} &P_{1}(x)+P_{2}(x)+P_{3}(x),
\ea\ee
where $\Phi$  stands for the standard normalized fundamental solution of Laplace equation in $\mathbb{R}^{3}$.\\
Since $\phi(x)=1, $ where $x\in B(\mu)$  ($0<\mu\leq\f{\rho}{2\sqrt{6}} $), we have
\[
\Delta(P_{2}(x)+P_{3}(x))=0.
\]
 According to the interior  estimate of harmonic function
and the H\"older  inequality, we thus have, for every
$ x_{0}\in B(\f{\rho}{4\sqrt{6}} )$,
\be\ba
|\nabla (P_{2}+P_{3})(x_{0})|&\leq \f{C}{\rho^{4}}\|(P_{2}+P_{3})\|_{L^{1}(B_{x_{0}}(\f{\rho}{4\sqrt{6}}))}
\\
&\leq \f{C}{\rho^{4}}\|(P_{2}+P_{3})\|_{L^{1}(B(\f{\rho}{2\sqrt{6}}))}\\
&\leq \f{C}{\rho^{4}}\rho^{3(1-\f{1}{q})} \|(P_{2}+P_{3})\|_{\wgr{L^{3/2}}(B(\f{\rho}{2\sqrt{6}}))}.\label{lem2.4.2.15}
\ea\ee
We infer from \eqref{lem2.4.2.15} that
\be\wgr{\|\nabla (P_{2}+P_3)\|^{3/2}_{L^{\infty}(B(\f{\rho}{4\sqrt{6}}))}\leq C \rho^{-9/2}\|P_2+P_3\|^{3/2}_{L^{3/2}(B(\f{\rho}{2\sqrt{6}}))}}.\label{lem2.4.2.16}\ee
Using the  mean value theorem  and \eqref{lem2.4.2.16} , for any  $\mu\leq \f{\rho}{4\sqrt{6}}$, we arrive at
\be\ba\label{lem2.4.2.17}
\|(P_{2}+P_{3})-\overline{(P_{2}+P_{3})}_{\mu}\|^{3/2}_{L^{3/2}(B(\mu))}\leq&
C\mu^{3} \|(P_{2}+P_{3})-\overline{(P_{2}+P_{3})}_{\mu}\|^{\wgr{3/2}}_{L^{\infty}(B(\mu))}\\
\leq& C
\mu^{3} (2\mu)^{\wgr{3/2}}\|\nabla (P_{2}+P_{3})\|^{\wgr{3/2}}_{L^{\infty}(B(\f{\rho}{4\sqrt{6}}))}\\
\leq& C\Big(\f{\mu}{\rho}\Big)^{\f{9}{2}}\|(P_{2}+P_{3})\|^{3/2}_{L^{3/2}
(B(\f{\rho}{2\sqrt{6}}))}.
\ea\ee
By time integration, we get
$$
\|(P_{2}+P_{3})-\overline{(P_{2}+P_{3})}_{\mu}\|
^{\f{3}{2}}_{L^{\f{3}{2}}(Q(\mu))}\leq C\Big(\f{\mu}{\rho}\Big)^{ \f{9}{2}}
\|(P_{2}+P_{3})\|^{\f{3}{2}}_{L^{\f{3}{2}}(Q(\f{\rho}{2\sqrt{6}}))}.
$$
As $(P_{2}+P_{3})-((P_{2}+P_{3}))_{B(\f{\rho}{2\sqrt{6}})}$ is also a Harmonic function  on $B(\f{\rho}{2\sqrt{6}})$, we deduce taht
$$\ba
&\|(P_{2}+P_{3})-\overline{(P_{2}+P_{3})}_{\mu}\|
^{\f{3}{2}}_{L^{3/2}(Q(\mu))}
\\
\leq & C\Big(\f{\mu}{\rho}\Big)^{ \f{9}{2}}
\|(P_{2}+P_{3})-\overline{(P_{2}+P_{3})}_{\f{\rho}{2\sqrt{6}}}\|
^{\f{3}{2}}
_{L^{\f{3}{2}}(Q(\f{\rho}{2\sqrt{6}}))}.
\ea$$
The triangle inequality guarantees that
$$\ba
&\|(P_{2}+P_{3})-\overline{(P_{2}+P_{3})}_{(\f{\rho}{2\sqrt{6}})}\|_{L^{\f{3}{2}}(Q(\f{\rho}{2\sqrt{6}}))}\\
\leq& \|\Pi-\overline{\Pi}_{(\f{\rho}{2\sqrt{6}})}\|_{L^{\f{3}{2}}(Q(\f{\rho}{2\sqrt{6}}))}
+\|P_{1}-\overline{P_{1}}_{(\f{\rho}{2\sqrt{6}})}\|_{L^{\f{3}{2}}(Q((\f{\rho}{2\sqrt{6}})))}
\\
\leq& \wgr{C}\|\wred{\Pi-\overline{\Pi}_{\rho}}\|_{L^{\f{3}{2}}(Q(\f{\rho}{2\sqrt{6}}))}
+\wgr{C}\|P_{1}\|_{L^{\f{3}{2}}(Q(\f{\rho}{2\sqrt{6}}))},
\ea$$
which leads to that
\be\label{p2rou}\ba
&\|(P_{2}+P_{3})-\overline{(P_{2}+P_{3})}_{\mu}\|
^{\f{3}{2}}_{L^{\f{3}{3}}(Q(\mu))}\\
\leq& C\Big(\f{\mu}{\rho}\Big)^{ \f{9}{2}}\Big(\|\Pi-\overline{\Pi}
_{(\rho)}\|^{\f{3}{2}}_{L^{\f{3}{2}}(Q(\f{\rho}{2\sqrt{6}}))}
+\|P_{1}\|^{\f{3}{2}}_{L^{\f{3}{2}}(Q(\f{\rho}{2\sqrt{6}}))}\Big)
\\
\leq& C\Big(\f{\mu}{\rho}\Big)^{ \f{9}{2}}\Big(\|\Pi-\overline{\Pi}
_{(\rho)}\|^{\f{3}{2}}_{L^{\f{3}{2}}(Q(\rho))}
+\|P_{1}\|^{\f{3}{2}}_{L^{\f{3}{2}}(Q(\f{\rho}{2\sqrt{6}}))}\Big).
\ea
\ee
By virtue of the H\"older inequality and the argument in \eqref{key2.9}, we get
$$\ba
&\iint_{Q(\rho/\sqrt{6})}|u-\bar{u}_{{\rho_{/\sqrt{6}}}}|^{3}dxds\\
\leq& C\iint_{Q(\rho/\sqrt{6})}|\wred{u- \overline{u}_{\rho/\sqrt{3}} }|^{3}dxds\\
 \leq& C \|u- \overline{ u_{\rho} }\|_{L^{p}L^{\overrightarrow{q}}(Q(\rho))}  \|\nabla u\|_{L^{2} (Q(\rho))}^{2-\f{2}{p} } \|u\|^{\f{2}{p}}_{L^{\infty}L^{2}(Q(\rho))}.
 \ea$$
The classical Calder\'on-Zygmund Theorem and the latter inequality implies that
\be\label{lem2.4.2}\ba
\iint_{Q(\f{\rho}{2\sqrt{6}})}|P_{1}(x)|^{\f{3}{2}}dxds
\leq& C \iint_{Q(\f{\rho}{\sqrt{6}})}|u-\overline{u}_{\rho/\sqrt{6}}|^{3}
dx\\
\leq&  C  \|u- \overline{ u_{\rho} }\|_{L^{p}L^{\overrightarrow{q}}(Q(\rho))}  \|\nabla u\|_{L^{2} (Q(\rho))}^{2-\f{2}{p} } \|u\|^{\f{2}{p}}_{L^{\infty}L^{2}(Q(\rho))},
 \ea\ee
and
\be\label{lem2.4.3}\ba
\iint_{Q(\mu)}|P_{1}(x)|^{\f{3}{2}}dx& \leq C  \|u- \overline{ u_{\rho} }\|_{L^{p}L^{\overrightarrow{q}}(Q(\rho))}  \|\nabla u\|_{L^{2} (Q(\rho))}^{2-\f{2}{p} } \|u\|^{\f{2}{p}}_{L^{\infty}L^{2}(Q(\rho))}.
 \ea\ee
The \wgr{inequalities} \eqref{p2rou}-\eqref{lem2.4.3} \wgr{allow} us to deduce that
\be\label{lem2.3}
\ba
\iint_{Q(\mu)}|\Pi-\Pi_{\mu}|^{\f{3}{2}}dxds \leq& C\iint_{Q(\mu)}
|P_{1}-(P_{1})_{\mu}|^{\f{3}{2}}
+\wgr{\big|P_{2}+P_{3}-(P_{2}+P_{3})_{\mu}\big|^{\f{3}{2}}}dx \\
\leq& C \|u- \overline{ u_{\rho} }\|_{L^{p}L^{\overrightarrow{q}}(Q(\rho))}  \|\nabla u\|_{L^{2} (Q(\rho))}^{2-\f{2}{p} } \|u\|^{\f{2}{p}}_{L^{\infty}L^{2}(Q(\rho))}
 \\
& +C\left(\f{\mu}{\rho}\right)
^{\f{9}{2}}\int_{B(\rho)}|\Pi-\Pi_{\rho}|^{\f{3}{2}}.
\ea
\ee
We readily get
\be\label{lem2.42}
\ba
\f{1}{\mu^2}\iint_{Q(\mu)}|\Pi-\Pi_{\mu}|^{\f{3}{2}}
\leq&  \wred{C\f{1}{\mu^2}}\|u- \overline{ u_{\rho} }\|_{L^{p}L^{\overrightarrow{q}}(Q(\rho))}  \|\nabla u\|_{L^{2} (Q(\rho))}^{2-\f{2}{p} } \|u\|^{\f{2}{p}}_{L^{\infty}L^{2}(Q(\rho))}\\
& +C\left(\f{\mu}{\rho}\right)
^{\f{5}{2}}
\f{1}{\rho^{2}}\iint_{Q(\rho)}|\Pi-\overline{\Pi}_{\rho}|^{\f{3}{2}}dx,
\ea
\ee
which leads to
\begin{align}
D_{3/2}(\mu)\leq &
C\left(\f{\rho}{\mu}\right)
^{2}\wred{E_{p,\overrightarrow{q}}(\rho)E_{\ast}(\rho)^{1-\f{1}{p}} E^{\f{1}{p}}(\rho)}
+C\left(\f{\mu}{\rho}\right)^{\f{5}{2}}D_{3/2}(\rho)\wgr{.}\label{presure4}
\end{align}
The proof of this lemma is   completed.
\end{proof}
\begin{proof}[Proof of Theorem \ref{thcw1}]
\wred{From} \eqref{spe1}, we know that
 there is a constant $\varrho_0$ such that, for any $\varrho\leq \varrho_{0}$,
$$
\wred{\varrho^{1-\f{2}{p}-\f{1}{q_1}-\f{1}{q_2}-\f{1}{q_3}}
\|u-\overline{ u_{\varrho} }\|_{L_{t}^{p}L_{1}^{q_1}L_{2}^{q_2}L_{3}^{q_3}(Q(\varrho))}}\leq\varepsilon_{1}.
$$
By the Young inequality and local energy inequality \eqref{loc}, we have
\be\label{eq:88}\ba
E(\rho)+E_{\ast}(\rho)\leq& C\Big[E^{2/3}_{3}(2\rho)+E_{3}(2\rho)
+\wred{D_{3/2}}(2\rho)\Big]\\
\leq& C\Big[1+E_{3}(2\rho)
+\wred{D_{3/2}}(2\rho)\Big].
\ea\ee
From \wgr{\eqref{eq:88} and} \eqref{ineq2/2} in Lemma \ref{ineq}, we see that, for
$2\sqrt{6}\mu\leq\rho$,
\be\ba
E_{3}(\mu)\leq&C \left(\dfrac{\rho}{\mu}\right)^{2}
\wgr{E_{p,\overrightarrow{q}}}(\rho/2)\wred{E_{\ast}(\rho/2)^{1-\f{1}{p}} E^{\f{1}{p}}(\rho/2)}
    +C\left(\dfrac{\mu}{\rho}\right)E_{3}(\rho/2)\\
    \leq&C \left(\dfrac{\rho}{\mu}\right)^{2}
\wgr{E_{p,\overrightarrow{q}}}(\rho/2)\B( 1+E_{3}(\rho)
+\wred{D_{3/2}}(\rho)\B)
    +C\left(\dfrac{\mu}{\rho}\right)E_{3}(\rho/2)
\\
    \leq&C \left(\dfrac{\rho}{\mu}\right)^{2}
\wgr{E_{p,\overrightarrow{q}}}(\rho )\B( 1+E_{3}(\rho)
+\wred{D_{3/2}}(\rho)\B)
    +C\left(\dfrac{\mu}{\rho}\right)E_{3}(\rho ).\label{3.2}    \ea\ee
It follows form  \eqref{pe} in Lemma \ref{presure} that, for $8\sqrt{6}\mu\leq\rho$,
\be\label{3.3}
D_{3/2}(\mu)\leq
C\left(\f{\rho}{\mu}\right)
^{2}\wgr{E_{p,\overrightarrow{q}}}(\rho)\B(1+E_{3}(\rho)
+D_{3/2}(\rho)\B)
+C\left(\f{\mu}{\rho}\right)^{\f{5}{2}}D_{3/2}(\rho).
\ee
Before going further, we set
$$F(\mu)= E_{3}(\mu)+\wred{D_{3/2}(\mu)}.$$
With the help of \eqref{3.2}  and \eqref{3.3}, we conclude that
$$\ba
F(\mu)\leq&
C
\left(\dfrac{\rho}{\mu}\right)^{2}
\wgr{E_{p,\overrightarrow{q}}}(\rho)F(\rho)
    \wgr{+C}\left(\dfrac{\rho}{\mu}\right)^{2}
\wred{E_{p,\overrightarrow{q}}}(\rho)
+C
\left(\dfrac{\mu}{\rho}\right)
 F(\rho)\\
\leq & C_{1}\lambda^{-2}\varepsilon_{1} F(\rho)+
C_{2}\lambda^{-2} \varepsilon_{1}  +C_{3}\lambda F(\rho),
\ea$$
\wgr{where  $\lambda=\f{\mu}{\rho}\leq \f{1}{8\sqrt{6}}$} and $\rho\leq   \varrho_{0}  $.\\
Choosing $\lambda,~\varepsilon_{1}$ such that $\wgr{\theta}=2C_{3}\lambda<1$ and $\varepsilon_{1}=\min\{ \f{\wgr{\theta}\lambda^{2}}{2C_{1}} ,
 \f{(1-\wgr{\theta})\lambda^{2}\varepsilon}{2C_{2}\lambda^{-2}}
 \}$ \wgr{where $\varepsilon$ is the constant in \eqref{CKN}}, we see that
\be\label{iter}
F(\lambda\rho)\leq \wgr{\theta}F(\rho)+C_{2}\lambda^{-2} \varepsilon_{1}. \ee
We iterate  $\eqref{iter}$ to get
 \[
F(\lambda^{k}\rho)\leq \wgr{\theta}^{k}F(\rho)+\f{1}{2}\lambda^{2}\varepsilon. \]
According to the definition of $F(r)$, for a fixed $\varrho_{0}>0$, we know that there exists a positive number $ K_{0}$~such that
$$\wgr{\theta}^{K_{0}}F(\varrho_{0})\wred{\leq}\f{M(\|u\|_{L^{\infty}L^{2}},\|u\|_{L^{2}W^{1,2}},
\|\Pi\|_{L^{3/2}L^{3/2}})}{\varrho_{0}^{2}}\wgr{\theta}^{K_{0}}
\leq\dfrac{1}{2}\varepsilon \lambda^{2}.$$
We denote $\varrho_{1}:=\lambda^{K_{0}}\varrho_{0}$. Then,  for all $0<\varrho\leq \varrho_{1}$ , $\exists k\geq
K_{0}$,~such that $\lambda^{k+1}\varrho_{0}\leq \varrho\leq \lambda^{k} \varrho_{0}$, there holds
\[
 \begin{aligned}
& E_{3}(\varrho)+D_{3/2}(\varrho)\\
=&\frac{1}{\varrho^{2}}\iint_{Q(\varrho)}|  u|^3dxdt+
\frac{1}{\varrho^{2}}\iint_{Q(\varrho)}|\Pi-\overline{\Pi}_{\varrho}|^{\frac{3}{2}}dxdt\\
\leq&  \frac{1}{(\lambda^{k+1}\varrho_{0})^{2}}\iint_{Q(\lambda^{k}\varrho_{0})}|  u|^3dxdt
 +
\frac{1}{(\lambda^{k+1}\varrho_{0})^{2}}\iint_{Q(\lambda^{k}\varrho_{0})}|\Pi-\overline{\Pi}_{\lambda^{k}\varrho_{0}}|^{\frac{3}{2}}dxdt
\\
\leq & \f{1}{\lambda^{2}}F(\lambda^{k}\varrho_{0}) \\
\leq &\f{1}{\lambda^{2}}(\wgr{\theta}^{k-K_{0}}\wgr{\theta}^{K_{0}}
F(\varrho_{0})+\f{1}{2}\lambda^{2}\varepsilon )\\
\leq &\varepsilon.
 \end{aligned}
\]
This together with  \eqref{CKN} completes the \wred{proof of} Theorem \ref{thcw1}.
\end{proof}

\section{ Regularity criteria  at one scale }\label{sec5}
Before going further, we write
 \be
\alpha=\f{2}{\f{2}{p}+\f{1}{\overrightarrow{q}}}.
\label{aerfa1}\ee
By virtue of the H\"older inequality, we just need consider the case that $\alpha$
is very close to 1 to show Theorem \ref{anatones}. Therefore, for any $1<q_{i}<\infty$, we have
\be
q_{i} \leq \f{2\alpha}{\alpha-1}.
\label{q>2}\ee

\begin{lemma}\label{zc2}  Let  $\alpha$ \wgr{be given}  in \eqref{aerfa1}. For $0<\xi<\eta$, there is an absolute constant $C$   such that \wred{for $\varrho=\frac{\xi+3\eta}{4}$}
\be\ba\label{zw}
 \|u\|_{L^{3}(Q(\varrho))}^{3}\leq C
\eta^{\f{3(\alpha-1)}{2}}\|u\|^{\alpha}_{\wred{L_t^{p}L_x^{\overrightarrow{q}}(Q(\eta))}} \Big\{\big[1+
\f{\eta^{{\f{7-3\alpha-\f{4\alpha}{p}}{2}}}}
{(\eta-\varrho)^{\f{7-3\alpha-\f{4\alpha}{p}}{2}}}\big]
\|u\|_{\wred{L_t^{\infty}L_x^{2}} (Q(\eta))}^{3-\alpha}+\|\nabla u\|_{L^{2}(Q(\eta))}^{3-\alpha }\Big\}.
\ea\ee
\end{lemma}
\begin{proof}
By the interpolation inequality in Lemma \ref{interi}, we find
\be\label{4.2}
\|u\|_{L^{3}(B(\varrho))}\leq \|u\|^{^{\f{\alpha}{3}}}_{L^{\overrightarrow{q}}(B(\varrho))}
\|u\|^{\f{3-\alpha}{3}}_{L^{\overrightarrow{t}}(B(\varrho))}.
\ee
This together with \eqref{q>2} implies that
$$
t_{j}=\f{3-\alpha}{1-\f{\alpha}{q_{j}}}\geq2, ~~\text{and}~~ \f{1}{\overrightarrow{t}}=\f{1+\f{2\alpha}{p}}{3-\alpha}\geq\f12.
$$
As a consequece, we can apply \eqref{locanlast} to obtain
$$\ba
\|u\|_{L^{\overrightarrow{t}}(B(\varrho))}
&\leq \wred{C
\|\nabla u \|_{L^{2}(B(\eta))}^{\f{3}{2}-\f{1}{\overrightarrow{t}}}
\|u\|_{L^{2}(B(\eta))}^{\f{1}{\overrightarrow{t}}-\f12}
+C(\eta-\varrho)^{\f{1}{\overrightarrow{t}}-\f{3}{2}}
\| u \|_{L^{2}(B(\eta))}.}
\ea$$
Plugging this into \eqref{4.2}, we know that
$$\ba
\|u\|^{3}_{L^{3}(B(\varrho))}\leq & \wred{C}\|u\|^{\alpha}_{L^{\overrightarrow{q}}(B(\eta))}
(\|\nabla u \|_{L^{2}(\eta)}^{\f{3}{2}-\wred{\f{1}{\overrightarrow{t}}}}
\|u\|_{L^{2}(B(\eta))}^{\wred{\f{1}{\overrightarrow{t}}}-\f12}
+(\eta-\varrho)^{\f{1}{\overrightarrow{t}}-\f{3}{2}}
\| u \|_{L^{2}(B(\eta))})^{3-\alpha} \\
\leq& \wred{C}\|u\|^{\alpha}_{L^{\overrightarrow{q}}(\wred{B(\eta)})}
(\|\nabla u \|_{L^{2}(B(\eta))}^{\f{7-3\alpha-\f{4\alpha}{p}}{2 }}
\|u\|_{L^{2}(B(\eta))}^{\f{\alpha-1+\f{4\alpha}{p}}{2 }}+
(\eta-\varrho)^{\wred{-\f{7-3\alpha-\f{4\alpha}{p}}{2 }}}
\| u \|_{L^{2}(B(\eta))}^{3-\alpha})\wred{.}
\ea$$
According to time integration and the Holder inequality, we see that
$$\ba
&\|u\|^{3}_{L^{3}(Q(\varrho))}\\
\leq & \wred{C\eta^{\f{3(\alpha-1)}{2}}\|u\|^{\alpha}_{L_t^{p}L_x^{\overrightarrow{q}}(Q(\eta))}} \B[\|\nabla u\|_{L^{2}(Q(\rho))}^{\f{7-3\alpha-\f{4\alpha}{p} }{2}}\|u\|^{\f{\alpha-1+\f{4\alpha}{p}}{2}}_{\wred{L_t^{\infty}L_x^{2}}(Q(\eta))}
+\f{\eta^{{\f{7-3\alpha-\f{4\alpha}{p}}{2}}}}{(\eta-\varrho)^{\f{7-3\alpha-\f{4\alpha}{p}}{2}}}
\|u\|^{3-\alpha}_{\wred{L_t^{\infty}L_x^{2}}(Q(\eta))}\B]
\\ \leq& \wred{C}\eta^{\f{3(\alpha-1)}{2}}\|u\|^{\alpha}_{\wred{L_t^{p}L_x^{\overrightarrow{q}}}(Q(\eta))} \Big\{\big[1+
\f{\eta^{{\f{7-3\alpha-\f{4\alpha}{p}}{2}}}}
{(\eta-\varrho)^{\f{7-3\alpha-\f{4\alpha}{p}}{2}}}\big]
\|u\|_{\wred{L_t^{\infty}L_x^{2}} (Q(\eta))}^{3-\alpha}+\|\nabla u\|_{L^{2}(Q(\eta))}^{3-\alpha }\Big\}.
\ea
$$
\end{proof}
\wred{Now}, \wgr{we prove} Theorem \ref{anatones}. Since the proof is parallel to the one used in \cite{[HWZ]}, we just sketch the proof.
\begin{proof}[Proof of Theorem \ref{anatones}] It  suffices to proof the following inequality, for any $R>0$,
\be\ba \label{key ineq}
&\|u\|^2_{L_{t}^{\infty}L_{x}^{2}(Q(R/2))}+\|\nabla u\|^2_{L^{2}(Q(R/2))} \\ \leq&    \f{C}{R^{( 4-3\alpha)/\alpha}} \|u\|_{L_{t}^{p}L^{\overrightarrow{q}}_{x}(Q(R))} ^{2} +  \f{C}{R^{( 5-3\alpha)/(\alpha-1)}} \|u\|_{L_{t}^{p}L^{\overrightarrow{q}}_{x}(Q(R))} ^{\f{2\alpha}{\alpha-1}} + \wgr{\f{C}{R^{5}}}\|\Pi\|^{2}_{L^{1}(Q(R))}.
\ea\ee
Indeed, consider $0<R/2\leq \xi<\f{3\xi+\eta}{4}<\f{\xi+\eta}{2}<\f{\xi+3\eta}{4}\wgr{<\eta\leq R}$. Let $\phi(x,t)$ be non-negative smooth function supported in $Q(\f{\xi+\eta}{2})$ such that
$\phi(x,t)\equiv1$ on $Q(\f{3\xi+\eta}{4})$,
$|\nabla \phi| \leq  C/(\eta-\xi) $ and $
|\nabla^{2}\phi|+|\partial_{t}\phi|\leq  C/(\eta-\xi)^{2} .$

The local energy inequality \eqref{loc}, the decomposition of  pressure  in Lemma \ref{lem2} and the H\"older inequality ensure that
\begin{align}
 &\int_{B(\f{\eta+\xi}{2})} |u(x,t)|^{2} \phi(x,t) dx
 +2\iint_{Q(\f{\eta+\xi}{2})}
  |\nabla u|^{2}\phi  dxds\nonumber\\\leq &  \f{C\wgr{\eta^{5/3}}}{(\eta-\xi)^{2}}\B(\iint_{Q(\wgr{\f{\xi+3\eta}{4}})} |u|^{3}
dxds\Big)^{2/3}+\f{C}{(\eta-\xi)} \iint_{Q(\wred{\f{\xi+3\eta}{4}})}  |u|^{3} dxds\nonumber\\&+ \f{C\eta^{3}}{(\eta-\xi)^{4}}  \|u\|^{3}_{L^{3}(Q(\wgr{\f{\xi+3\eta}{4}}))}+\f{C\eta^{3/2}}{(\eta-\xi)^{4}} \|  \Pi\|_{L^{1} (Q(\eta))}  \|u\|_{\wred{L_t^{\infty}L_x^2}(Q(\eta)))}\nonumber\\
=&:I+II+III\wred{+IV.}\label{last3}
 \end{align}
Combining
  \eqref{zw} and the Young inequality, we obtain
\begin{align}\nonumber I\leq & \f{C\eta^{3+\f{2}{\alpha} }}{(\eta-\xi)^{6/\alpha}} \|u\|^{2}_{L_{t}^{p}L^{\overrightarrow{q}}(Q(\eta))}  +\f{C\eta^{3+\f{2}{\alpha} }}{(\eta-\xi)^{6/\alpha}} \big[1+
\f{\eta^{{\f{7-3\alpha-\f{4\alpha}{p}}{2}}}}
{(\eta-\xi)^{\f{7-3\alpha-\f{4\alpha}{p}}{2}}}\big]^{\f{2}{\alpha}}\|u\|^{2}_{\wred{L_{t}^{p}L_x^{\overrightarrow{q}}}(Q(\eta))}  \\&+\f{1}{6}\B(\|u\|_{\wred{L_t^{\infty}L_x^2} (Q(\eta))}^{2}+\|\nabla u\|_{L^{2}(Q(\eta))}^{2}\B),\label{last2}\\
\nonumber II
\leq & \f{C\eta^{3}}{(\eta-\xi)^{\f{2}{\alpha-1}}}\|u\|^{\f{2\alpha}{\alpha-1}}_{\wred{L_{t}^{p}L_x^{\overrightarrow{q}}}(Q(\eta))}
\B\{1+\big[1+
\f{\eta^{{\f{7-3\alpha-\f{4\alpha}{p}}{2}}}}
{(\eta-\varrho)^{\f{7-3\alpha-\f{4\alpha}{p}}{2}}}\big]^{\f{2}{\alpha-1}}\B\} \\&
+\f{1}{6}\B(\|u\|_{\wred{L_t^{\infty}L_x^2}(Q(\eta))}^{2}+\|\nabla u\|_{L^{2}(Q(\eta))}^{2}\B),\\ \nonumber
III\leq& \f{C\eta^{\f{3(\alpha+1)}{\alpha-1} }}{(\eta-\xi)^{\f{8}{(\alpha-1)}}}
\|u\|^{\f{2\alpha}{\alpha-1}}
_{\wred{L_{t}^{p}L_x^{\overrightarrow{q}}}(Q(\eta))}
\B\{1+\big[1+
\f{\eta^{{\f{7-3\alpha-\f{4\alpha}{p}}{2}}}}
{(\eta-\varrho)^{\f{7-3\alpha-\f{4\alpha}{p}}{2}}}\big]^{\f{2}{\alpha-1}}\B\} \\&
+\f{1}{6}\B(\|u\|_{\wred{L_t^{\infty}L_x^2}(Q(\eta))}^{2}+\|\nabla u\|_{L^{2}(Q(\eta))}^{2}\B).
\end{align}
Using the Young inequality again, we conclude that
\be IV\leq   \f{C\eta^{3 }}{(\eta-\xi)^{8}} \|  \Pi\|^{2}_{L^{1} (Q(\eta))} +\f{1}{6} \|u\|^{2}_{L^{2,\infty}(Q(\eta)))}.\label{last1}\ee
After plugging  \eqref{last2}-\eqref{last1} into \eqref{last3}, we apply
  iteration Lemma \cite[Lemma V.3.1,   p.161]{[Giaquinta]} to finish the proof.
\end{proof}
		
\section{ Regularity criteria  at one scale \wred{without pressure}}\label{sec4}
As explained \wred{in Section 1}, the proof of Theorem \ref{[Leraybackself]} reduces to the proof of \wred{Theorem \ref{the1.5}.} Thanks to \eqref{wolfchae}  or the results in \cite{[Wolf1],[JWZ]}, we just need to prove the following Caccioppoli type inequality \wred{given in Lemma \ref{zc22}}. We write
\be\label{aerfa2}
\wred{\alpha=\f{2}{\f{1}{\overrightarrow{q}}}.}
\ee
By the H\"older inequality, we just need consider the case that $\alpha$
is very close to 1. Therefore, for any $1<q_{i}<\infty$, we have
\be
\f{2}{3}\geq (\f{2}{3}-\f{1}{q_{i}})\alpha.
\label{q>24}\ee
\wred{For example, $\overrightarrow{q}=(1000,\f{1000}{990},\f{1000}{990})$ do not hold for \eqref{q>24}, by the H\"older inequality, we just consider $(1000,\f{1000}{999},\f{1000}{999})$ to ensure that \eqref{q>24} is valid.}
\begin{lemma}\label{zc22}  Let  $\alpha$ \wgr{be given}  in \eqref{aerfa2}. For \wred{any $R>0$}, there is an absolute constant $C$   such that
\be\ba\label{zwl}
&\|u\|^{2}_{\wred{L_t^{3}L_x^{\f{18}{5}}}(Q(\f{R}{2}))}+ \|\nabla u\|^{2}_{L^{2}(Q(\f{R}{2}))}\\
\leq &
  CR^{\f{3\alpha-4}{\alpha}}\| u\|^{2}_{\wred{L_t^{\infty}L_x^{\overrightarrow{q}}}(Q(R))}
  +CR^{\f{5\alpha-8}{2(\alpha-1)}}\| u\|^{\f{3\alpha}{2(\alpha-1)}}_{\wred{L_t^{\infty}L_x^{\overrightarrow{q}}}(Q(R))}
  +CR^{\f{4\alpha-6}{\alpha}}\| u\|^{3}_{\wred{L_t^{\infty}L_x^{\overrightarrow{q}}}(Q(R))}.
\ea\ee
\end{lemma}
\begin{proof}
\wred{Consider} $0<R/2\leq r<\f{3r+\rho}{4}<\f{r+\rho}{2}<\rho\leq R$. Let $\phi(x,t)$ be non-negative smooth function supported in $Q(\f{r+\rho}{2})$ such that
$\phi(x,t)\equiv1$ on $Q(\f{3r+\rho}{4})$,
$|\nabla \phi| \leq  C/(\rho-r) $ and $
|\nabla^{2}\phi|+|\partial_{t}\phi|\leq  C/(\rho-r)^{2} .$

The interpolation inequality in \cite{[BP]} allows us to arrive at \be\label{4.24}
\|u\|_{L^{3}\wred{(B(\f{r+3\rho}{4}))}}\leq \|u\|^{^{\f{\alpha}{4-\alpha}}}_{L^{\overrightarrow{q}}(\wred{B(\f{r+3\rho}{4})})}
\|u\|^{\f{2(2-\alpha)}{4-\alpha}}_{L^{\overrightarrow{t}}(\wred{B(\f{r+3\rho}{4})})},
\ee
Combining this and  \eqref{q>24} yields that
$$
\f{1}{\overrightarrow{t}}=\f12, ~~\text{and}~~~t_{j}=\f{6q_{i}(2-\alpha)}{(4-\alpha)q_j-3\alpha}\geq2.
$$
Thence, inequality \eqref{locanlast} gives
$$\ba
\|u\|_{L^{\overrightarrow{t}}(\wred{B}(\f{r+3\rho}{4}))}&\leq
\|u-\wred{\overline{u}_{\rho}} \|_{L^{\overrightarrow{t}}(B(\f{r+3\rho}{4}))}
+\|\wred{\overline{u}_{\rho}}\|_{L^{\overrightarrow{t}}(B(\f{r+3\rho}{4}))}\\
&\leq\wred{C}
\|\nabla u \|_{L^{2}(\wred{B(\rho)})}
+\wred{C}\rho^{\f{1}{\overrightarrow{t}}-\f{1}{\overrightarrow{q}}}
\| u \|_{L^{\overrightarrow{q}}(\wred{B(\rho)})}.
\ea$$
\wred{Plugging this into \eqref{4.24}}, we infer that \wred{for $\theta=\frac{\alpha}{4-\alpha}$,}
$$\ba
\|u\|^{3}_{L^{3}B(\f{r+3\rho}{4})}&\leq \wred{C}\|u\|^{3\theta}_{L^{\overrightarrow{q}}(B(\f{r+3\rho}{4}))}
\|u\|^{3(1-\theta)}_{L^{\overrightarrow{t}}(B(\f{r+3\rho}{4}))} \\ &\leq\wred{C}
\|u\|^{3\theta}_{\wred{L^{\overrightarrow{q}}}(B(\f{r+3\rho}{4}))}\B[\|\nabla u\|_{L^{2}\wred{(B(\rho))}}+\rho^{\f{1}{\overrightarrow{t}}-\f{1}{\overrightarrow{q}}}\|u\|_{L^{\overrightarrow{q}}(\wred{B(\rho)})}
\B]^{3(1-\theta)}.
\ea$$
As a consequence, we infer \wred{that}
$$
\|u\|^{3}_{L^{3}(Q(\f{r+3\rho}{4}))}\leq \wred{C}\rho^{\f{4(\alpha-1)}{4-\alpha}}\|u\|^{\f{3\alpha}{4-\alpha}}_{\wred{L_t^{\infty}L_x^{\overrightarrow{q}}}(Q(\rho))} \|\nabla u\|_{L^{2}(Q(\rho))}^{\f{3(4-2\alpha)}{4-\alpha}}
+\wred{C}\rho^{\f{5\alpha-6}{\alpha}}\|u\|^{3}_{\wred{L_t^{\infty}L_x^{\overrightarrow{q}}}(Q(\rho))}.
 $$

Let $\nabla\Pi_{h}=\mathcal{W}_{3,B(\f{r+3\rho}{4})}(u)$, then,
from \eqref{ph}-\eqref{p2}, we have
\begin{align}
&\|\nabla \Pi_{h}\|_{L^{3}(Q(\f{r+3\rho}{4}))}\leq C\|u\|_{L^{3}(Q(\f{r+3\rho}{4}))},\label{wp1}\\
 &\|  \Pi_{1}\|_{L^{2}(Q(\f{r+3\rho}{4}))}\leq C\|\nabla u\|_{L^{2}(Q(\f{r+3\rho}{4}))},\label{wp2}\\
 &\|  \Pi_{2}\|_{L^{\f{3}{2}}(Q(\f{r+3\rho}{4}))}\leq C\|  |u|^{2}\|_{L^{\f{3}{2}}(Q(\f{r+3\rho}{4}))}.\label{wp3}
 \end{align}
Since $v=u+\nabla\Pi_{h}$, the H\"older inequality and \eqref{wp1} allows us to write
\begin{align}
 \iint_{Q(\rho)} \wred{| v|^2}\Big|  \Delta \phi^{4}+  \partial_{t}\phi^{4}\Big|    \leq& \f{C}{(\rho-r)^{2}}\iint_{Q(\f{r+\rho}{2})} |u|^{2}+|\nabla\Pi_{h}|^{2}
\nonumber\\\leq& \f{C\rho^{5/3}}{(\rho-r)^{2}}\B(\iint_{Q(\f{r+\rho}{2})} |u|^{3}+|\nabla\Pi_{h}|^{3}\B)^{\f{2}{3}}
\nonumber\\\leq& \f{C\rho^{5/3}}{(\rho-r)^{2}}\|u\|_{L^{3}(Q(\f{r+3\rho}{4}))}^{2}.\label{53.2} \end{align}
 It follows from  H\"older's inequality, $v=u+\nabla\Pi_{h}$ and  \eqref{wp1} that
 \begin{align}
 \iint_{Q(\rho)}|v|^{2}\phi^{3}u\cdot\nabla \phi
\leq \f{C}{(\rho-r)}
 \| u \|^{3}_{L^{3}(Q(\f{r+3\rho}{4}))}.
 \end{align}
According to interior estimate of harmonic function \eqref{h1} and \eqref{wp1}, we have
$$\ba\|\nabla^{2}\Pi_{h} \|_{L^{20/7}(Q(\f{r+\rho}{2}))}&\leq
\f{\wred{C} (r+\rho)  }{(\rho-r)^{ 2}}
\|\nabla\Pi_{h} \|_{L^{3}(Q( \f{r+3\rho}{4} ))}\\
&\leq
\f{ C\rho  }{(\rho-r)^{ 2}}
\|u \|_{L^{3}(Q( \f{r+3\rho}{4}))},
\ea$$
from which it follows that
\begin{align}
&\iint_{Q(\rho)} \phi^{4}( u\otimes v :\nabla^{2}\Pi_{h} )  \nonumber\\
\leq&
\| v\phi^{2}\|_{L^{3}(Q(\f{r+\rho}{2}))}
\| u \|_{L^{3}(Q(\f{r+\rho}{2}))}\|\nabla^{2}\Pi_{h} \|_{L^{3}(Q(\f{r+\rho}{2}))}
\nonumber\\
\leq &\f{ C\rho  }{(\rho-r)^{ 2}}
\|u \|^{3}_{L^{3}(Q( \f{r+3\rho}{4} ))}.
\end{align}
Taking advantage of the H\"older inequality, \eqref{wp2} and Young's inequality, we infer that
\begin{align}
\iint_{Q(\rho)} \phi^{3} \Pi_{1}v\cdot\nabla \phi
&\leq \f{C}{(\rho-r)}\| v\|_{L^{2}(Q(\f{r+\rho}{2}))}
\| \Pi_{1} \|_{L^{2}(Q(\f{r+\rho}{2}))} \nonumber\\
&\leq \f{C}{(\rho-r)^{2}}\| v\|^{2}_{L^{2}(Q(\f{r+\rho}{2}))}
+\f{1}{16}\| \Pi_{1} \|^{2}_{L^{2}(Q(\f{r+3\rho}{4}))} \nonumber\\
&\leq \f{C\rho^{5/3}}{(\rho-r)^{2}}\|u\|_{L^{3}(Q(\f{r+3\rho}{4}))}^{2}+\f{1}{16}\| \nabla u\|^{2}_{L^{2}(Q(\f{r+3\rho}{4}))}.
\end{align}
We conclude from the H\"older inequality \wred{and \eqref{wp3}   that}
\begin{align}
\iint_{Q(\rho)} \phi^{3} \Pi_{2}v\cdot\nabla \phi
 \leq \f{C}{(\rho-r)}\| v\phi^{2}\|_{L^{3}(Q(\f{r+\rho}{2}))}
\| \Pi_{2} \|_{L^{\f{3}{2}}(Q(\f{r+\rho}{2}))} \leq\f{ C }{(\rho-r)}
\|u \|^{3}_{L^{3}(Q(\f{r+3\rho}{4} ))}.\label{locp5}
\end{align}
Plugging \eqref{53.2}-\eqref{locp5} into local energy \wred{inequality \eqref{wloc1}, we infer} that
\begin{align}
\sup_{-\rho^{2}\leq t\leq0}\int_{B(\rho)}|v\phi^{2}|^2   + \iint_{Q(\rho)}\big|\nabla( v\phi^{2})\big|^2  \leq& \f{C\rho^{5/3}}{(\rho-r)^{2}}\|u\|_{L^{3}(Q(\f{r+3\rho}{4}))}^{2} +\f{ C\rho  }{(\rho-r)^{ 2}}
\|u \|^{3}_{L^{3}(Q( \f{r+3\rho}{4} ))}\nonumber\\&+\f{ C }{(\rho-r)}
\|u \|^{3}_{L^{3}(Q( \f{r+3\rho}{4} ))}\wred{+\f{1}{16}\| \nabla u\|^{2}_{L^{2}(Q(\f{r+3\rho}{4}))}.}\label{keyl}
  \end{align}
Applying the interior estimate of harmonic function
  \eqref{h1} and \eqref{wp1} implies that
$$\ba\|\nabla\Pi_{h}\|^{2}_{\wred{L_t^{3}L_x^{\f{18}{5}}}(Q(r))}&\leq \f{Cr^{\f{5}{3}}}{(\rho-r)^{2}}\|\nabla\Pi_{h}\|^{2}_{L^{3}Q(\wred{\f{r+3\rho}{4}})}\leq \f{Cr^{\f{5}{3}}}{(\rho-r)^{2}}\wred{\|u\|^{2}_{L^{3}Q(\f{r+3\rho}{4})},}
\ea$$
\wred{which together with the triangle inequality  and \eqref{keyl} ensures} that
$$\ba
 \|u\|^{2}_{\wred{L_t^{3}L_x^{\f{18}{5}}}(Q(r))} \leq& \|v\|^{2}_{\wred{L_t^{3}L_x^{\f{18}{5}}}(Q(r))}+\|\nabla\Pi_{h}\|^{2}_{\wred{L_t^{3}L_x^{\f{18}{5}}}(Q(r))}\\
\leq& C\B\{\|v\|_{\wred{L_t^{\infty}L_x^2}(Q(r))}^{2}+\|\nabla v\|_{L^{2}(Q(r))}^{2}\B\}+\f{Cr^{\f{5}{3}}}{(\rho-r)^{2}}\|u\|^{2}_{L^{3}Q(\f{r+3\rho}{4})}\\
\leq & \f{C\rho^{5/3}}{(\rho-r)^{2}}\|u\|_{L^{3}(Q(\f{r+3\rho}{4}))}^{2}+\f{ C\rho  }{(\rho-r)^{ 2}}
\|u \|^{3}_{L^{3}(Q( \f{r+3\rho}{4} ))} \\&+\f{ C }{(\rho-r)}
\|u \|^{3}_{L^{3}(Q( \f{r+3\rho}{4} ))}+\f{1}{16}\| \nabla u\|^{2}_{L^{2}(Q(\f{r+3\rho}{4}))}.
\ea$$
\wred{Using \eqref{h1} and \eqref{wp1}} once again, we know that
$$
 \|\nabla^{2} \Pi_{h}\|^{2}_{L^{2}(Q(r))}\leq \f{Cr^{3}}{(\rho-r)^{5}} \|\nabla\Pi_{h}\|^{2}_{L^{2}(Q(\f{r+\rho}{2}))}\leq \f{Cr^{3}\rho^{5/3}}{(\rho-r)^{5}} \| u \|^{2}_{L^{3}(Q(\f{r+3\rho}{4}))}.
$$
Combining  the triangle \wred{inequality and \eqref{keyl}} yields
\begin{align}
  \|\nabla u\|^{2}_{L^{2}(Q(r))}&\leq \|\nabla v\|^{2}_{L^{2}(Q(r))}+
 \|\nabla^{2} \Pi_{h}\|^{2}_{L^{2}(Q(r))}\nonumber\\\leq&
  \B\{\f{C\rho^{5/3}}{(\rho-r)^{2}}+\f{Cr^{3}\rho^{5/3}}{(\rho-r)^{5}} \B\} \| u \|^{2}_{L^{3}(Q(\f{r+3\rho}{4}))}\nonumber\\&
  +\B\{\f{ C\rho  }{(\rho-r)^{ 2}} +\f{ C }{(\rho-r)}\B\}
\|u \|^{3}_{L^{3}(Q( \f{r+3\rho}{4} ))}\wred{+\f{1}{16}\| \nabla u\|^{2}_{L^{2}(Q(\f{r+3\rho}{4}))}.}
\end{align}
Then, the following estimate holds
  $$\ba
&\|u\|^{2}_{\wred{L_t^{3}L_x^{\f{18}{5}}}(Q(r))}+ \|\nabla u\|^{2}_{L^{2}(Q(r))}
  \\\leq&
  \B\{\f{C\rho^{5/3}}{(\rho-r)^{2}}+\f{Cr^{3}\rho^{5/3}}{(\rho-r)^{5}} \B\}^{\f{4-\alpha}{\alpha}}\rho^{\f{8(\alpha-1)}{3\alpha}} \| u \|^{2}_{\wred{L_t^{\infty}L_x^{\overrightarrow{q}}}(Q(\rho))}\nonumber\\
  &+\B\{\f{C\rho^{5/3}}{(\rho-r)^{2}}+\f{Cr^{3}\rho^{5/3}}{(\rho-r)^{5}} \B\}
  \rho^{\f{2(5\alpha-6)}{3\alpha}}\|u\|^{2}_{\wred{L_t^{\infty}L_x^{\overrightarrow{q}}}(Q(\rho))}\nonumber\\&
  +\B\{\f{ C\rho  }{(\rho-r)^{ 2}} +\f{ C }{(\rho-r)}\B\}^{\f{4-\alpha}{2(\alpha-1)}}\rho^{2}
\|u \|^{\f{3\alpha}{2(\alpha-1)}}_{\wred{L_t^{\infty}L_x^{\overrightarrow{q}}}(Q( \rho))}\nonumber\\
&+\B\{\f{ C\rho  }{(\rho-r)^{ 2}} +\f{ C }{(\rho-r)}\B\}\rho^{\f{5\alpha-6}{\alpha}}\|u\|^{3}_{\wred{L_t^{\infty}L_x^{\overrightarrow{q}}}(Q(\rho))}+\f{3}{16}\| \nabla u\|^{2}_{L^{2}(Q(\rho))}.
 \ea $$
Now, we are in a position to apply  \wred{iteration lemma} \cite[Lemma V.3.1,   p.161 ]{[Giaquinta]} to   the latter  estimate to find that
$$\ba
&\|u\|^{2}_{\wred{L_t^{3}L_x^{\f{18}{5}}}(Q(\f{R}{2}))}+ \|\nabla u\|^{2}_{L^{2}(Q(\f{R}{2}))}\\
\leq &
  CR^{\f{3\alpha-4}{\alpha}}\| u\|^{2}_{\wred{L_t^{\infty}L_x^{\overrightarrow{q}}}(Q(R))}
  +CR^{\f{5\alpha-8}{2(\alpha-1)}}\| u\|^{\f{3\alpha}{2(\alpha-1)}}_{\wred{L_t^{\infty}L_x^{\overrightarrow{q}}}(Q(R))}
  +CR^{\f{4\alpha-6}{\alpha}}\| u\|^{3}_{\wred{L_t^{\infty}L_x^{\overrightarrow{q}}}(Q(R))}.
\ea$$
  This achieves the proof of \wred{this lemma}.
\end{proof}
		\section*{Acknowledgement}
The authors would like to express their sincere gratitude to Dr. Xiaoxin Zheng at the School
of Mathematics and Systems Science, Beihang University, for calling our attention to the
problem involving $\varepsilon$-regularity criteria in anisotropic Lebesgue spaces.
The research of Wang was partially supported by  the National Natural		Science Foundation of China under grant No. 11601492 and
the Youth Core Teachers Foundation of Zhengzhou University of
Light Industry.
\wgr{The research of Wu was partially supported by the National
Natural Science Foundation of China under grant No. 11771423 and No. 11671378.}

	\end{document}